\documentclass{article}

\usepackage{amsmath, amssymb}
\usepackage{graphicx, wrapfig} 
\usepackage{amsthm}
\usepackage{url}
\usepackage{color}
\usepackage{bm}
\usepackage{xspace}
\usepackage{hyperref}

\theoremstyle{plain}
\newtheorem{theorem}{Theorem}

\theoremstyle{definition}
\newtheorem{definition}{Definition}
\newtheorem{example}{Example}

\newtheorem{proposition}{Proposition}
\newtheorem{lemma}{Lemma}
\newtheorem{corollary}{Corollary}

\newcommand{\II}{\mbox{I\hspace{-0.5pt}I}\xspace}
\newcommand{\III}{\mbox{I\hspace{-0.5pt}I\hspace{-0.5pt}I}\xspace}
\newcommand{\IIIa}{\mbox{I\hspace{-0.5pt}I\hspace{-0.5pt}I$\alpha$}\xspace}

\title{The OU number and Reidemeister moves of type III for link diagrams}
\author{Naoki Sakata\thanks{Advanced Institute for Materials Research (AIMR), Tohoku University, 2-1-1 Katahira, Aoba-ku, Sendai, Miyagi, 980-8577, Japan. Email: sakata@casis.sakura.ne.jp},  
Ayaka Shimizu\thanks{Center for Soft Matter Physics, Ochanomizu University, 2-1-1, Otsuka, Bunkyo-ku, Tokyo, 112-8610, Japan. Email: shimizu.ayaka@ocha.ac.jp, shimizu1984@gmail.com} 
and Koya Shimokawa\thanks{Department of Mathematics, Ochanomizu University, 2-1-1, Otsuka, Bunkyo-ku, Tokyo, 112-8610, Japan. 
International Institute for Sustainability
with Knotted Chiral Meta Matter (WPI-SKCM$^2$),
Hiroshima University, 1-3-1 Kagamiyama, 
Higashi-Hiroshima, 739-8526, Japan. 
Research Institute for Science and Technology at Tokyo University of Science,  Division of Joint Research of Geometry and Natural Science, 1-3 Kagurazaka, Shinjuku-ku, Tokyo 162-8601, Japan.
Email: shimokawa.koya@ocha.ac.jp}}
\date{\today}

\begin{document}

\maketitle

\begin{abstract}
We introduce the non-self OU sequence and the OU number for link diagrams. 
Using these, we give a lower bound for the number of necessary Reidemeister moves of type III between two diagrams of the same link. 
\end{abstract}

\section{Introduction}

An {\it $r$-component link} $L= k_1 \cup k_2 \cup \dots \cup k_r$ is an embedding of $r$ circles in $S^3$ or $\mathbb{R}^3$. 
When each component $k_i$ is oriented, we say $L$ is {\it oriented}. 
When $r=1$, we also call $L$ a {\it knot}. 
A {\it diagram} $D=K_1 \cup K_2 \cup \dots \cup K_r$ of a link $L= k_1 \cup k_2 \cup \dots \cup k_r$ is a regular projection of $L$ onto $S^2$ or $\mathbb{R}^2$, where each intersection is a transverse double point with over/under information. 

It is classical in knot theory that two diagrams $D$, $D'$ represent the same link if and only if there exists a finite sequence of Reidemeister moves that deforms $D$ into $D'$, where the Reidemeister moves are the local transformations on diagrams of type RI, R\II or R\III shown in Figure~\ref{fig-r123}. 
\begin{figure}[htbp]
\centering
\includegraphics[width=7cm]{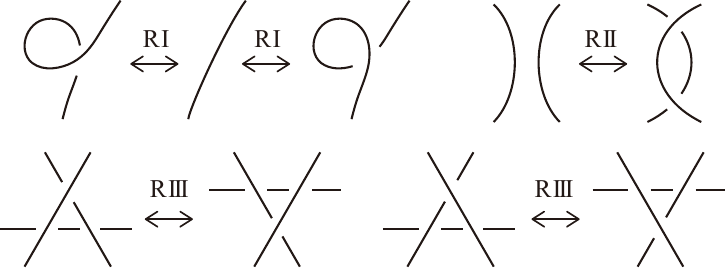}
\caption{Reidemeister moves. }
\label{fig-r123}
\end{figure}

The complexity of link diagrams plays a crucial role in understanding the dynamics of ring polymers.
In simulations of polymer melts composed of ring polymers, it has been observed that complex entanglement suppresses movement, leading to a glass transition known as ``topological glass''~\cite{MT}.
In such systems, any pair of ring polymers typically forms a trivial link.
Although they are mathematically unlinked, physically separating the two components requires specific spatial movements.
Consequently, the polymers form a complex entangled network that restricts individual mobility, which is expected to significantly affect the material properties.

To investigate mathematically how a two-component trivial link can be disentangled, it is natural to analyze the sequence of Reidemeister moves applied to its diagram.
Specifically, for two given diagrams of the same link, we can measure how ``distant'' they are by considering the minimum number of Reidemeister moves required to transform one to the other, or by determining the necessity of specific types of moves in the sequence.
Many studies have addressed this topic (see, for example, \cite{CESS, HH, M, O}). 

As far as R\III moves on diagrams of links with two or more components are concerned, the literature is limited. 
In \cite{CESS}, Carter, Elhamdadi, Saito and Satoh gave a pair of diagrams $D$ and $D'$ of the same 2-component link such that transforming $D$ into $D'$ in $\mathbb{R}^2$ requires three R\III moves, using Fox $n$-colorings. 
In \cite{J}, Jablonowski gave an infinite family of diagrams $D_{n,k}$ of a trivial 2-component link such that transforming $D_{n,k}$ into a diagram without crossings requires at least four R\III moves, employing a four-dimensional technique. 

In this paper, we introduce a numerical invariant of diagrams, called the {\it OU number}, for each component of a link diagram. 
Using this invariant, we give a lower bound for the number of necessary R\III moves of a particular type between two diagrams of the same link. 
Then we prove that for any natural number $n$, there exists a diagram $D$ of a trivial 2-component link such that transforming $D$ into a diagram without crossings requires at least $n$ R\III moves (Theorem~\ref{thm-infty}) using Jablonowski's diagram. 

\medskip

The OU sequence of a knot diagram was studied in \cite{FNS, HNSY, S}. 
In this paper, we introduce the {\it non-self OU sequence} for link diagrams. 
Let $D=K_1 \cup K_2 \cup \dots \cup K_r$ be an oriented link diagram on $S^2$. 
By traversing a component $K_i$ and recording the over/under information as ``$O$'' and ``$U$'' for only non-self crossings (i.e., crossings between distinct components) that are encountered, we obtain a cyclic sequence and call it the non-self OU sequence $f(K_i)$ of $K_i$. 
In Section~\ref{section-core-bridge}, we define the {\it OU number} $\Phi (K_i)$ of $K_i$, which can be easily computed from $f(K_i)$. 
Roughly speaking, the OU number $\Phi$ is the number of pairs of ``$OU$'' after repeatedly removing consecutive ``$OO$'' or ``$UU$'' pairs from $f(K_i)$. 
(See Section~\ref{section-core-bridge} for a precise definition.) 
For example, the link diagram $D=K_1 \cup K_2$ illustrated on the left-hand side in Figure \ref{fig-r3-ex} has the non-self OU sequences $f(K_1)=OUOU$, $f(K_2)=OOUU$ and OU numbers $\Phi (K_1)=2$, $\Phi (K_2)=0$. 

In Section~\ref{section-r3}, we divide R\III moves into two types, ``R\IIIa'' and the others, considering the layeredness of components. 
An R\III move is called an \emph{R\IIIa move} when the middle segment belongs to a different component from the other segments. 
The R\III move indicated in Figure~\ref{fig-r3-ex} is of type R\IIIa. 
\begin{figure}[htbp]
\centering
\includegraphics[width=6.5cm]{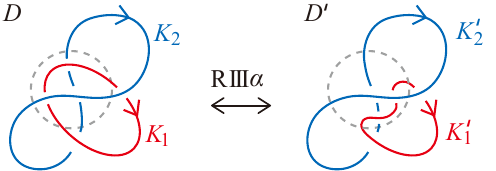}
\caption{An R\IIIa move.}
\label{fig-r3-ex}
\end{figure}

The R\IIIa move can play a role in detecting a ``threading'' under specific conditions.
Here, we define a threading in a two-component trivial link as a configuration where one component intersects the minimal surface bounded by the other component (cf.~\cite{SG}).
In Figure~\ref{fig-r3-ex}, $K_1$ can be regarded as threading $K_2$.
This interpretation arises because a portion of the spanning disk of $K_2$ is visually implied around the self-crossing of $K_2$.
Since performing a single R\IIIa move is necessary to separate $K_1$ from $K_2$, the non-self crossing effectively represents the point where $K_1$ penetrates the disk.

In Section~\ref{section-r3}, we prove the following theorem. 

\begin{theorem}
For each knot component $K_i$ of an \textup{(}oriented or unoriented\textup{)} link diagram $D=K_1 \cup K_2 \cup \dots \cup K_r$, the OU number $\Phi (K_i)$ is invariant under Reidemeister moves except for R\IIIa. 
\label{thm-r3a}
\end{theorem}

\noindent Theorem \ref{thm-r3a} implies that the OU number can be used as an ``R\IIIa-indicator'' for link diagrams. 
For example, we observe that any sequence of Reidemeister moves between the two link diagrams $D=K_1 \cup K_2$ and $D' = K'_1 \cup K'_2$ in Figure \ref{fig-r3-ex} requires at least one R\IIIa move because $\Phi (K_1)=2$, $\Phi (K_2)=0$ and $\Phi (K'_1)=0$, $\Phi (K'_2)=0$. 
Moreover, we discuss the lower bound for the number of necessary R\IIIa moves, and prove the following theorem in Section~\ref{section-r3}. 

\begin{theorem}
Let $D=K_1 \cup K_2 \cup \dots \cup K_r$ and $D'= K'_1 \cup K'_2 \cup \dots \cup K'_r$ be diagrams of the same unoriented {\it ordered} link. 
Any sequence of Reidemeister moves that transforms $D$ into $D'$ must contain at least 
\begin{align*}
\frac{1}{2} \sum_{i=1}^r | \Phi (K'_i) - \Phi (K_i)|
\end{align*}
Reidemeister moves of type R\IIIa. 
\label{thm-ordered}
\end{theorem}

\noindent When the knot components are not distinguished in a link diagram, we have the following. 

\begin{theorem}
Let $D=K_1 \cup K_2 \cup \dots \cup K_r$ and $D'= K'_1 \cup K'_2 \cup \dots \cup K'_r$ be diagrams of the same unoriented {\it unordered} link. 
Any sequence of Reidemeister moves that transforms $D$ into $D'$ must contain at least 
\begin{align*}
\frac{1}{2} \left| \sum_{i=1}^r \Phi (K'_i) - \sum_{i=1}^r \Phi (K_i) \right|
\end{align*}
Reidemeister moves of type R\IIIa. 
\label{thm-unordered}
\end{theorem}

\noindent We call a diagram of an $r$-component link with no (self or non-self) crossings a {\it trivial diagram}. 
We prove the following theorem in Section~\ref{section-example}. 

\medskip 
\begin{theorem}
For any natural number $n$, there exists a diagram $D$ of a trivial $2$-component link such that transforming $D$ into a trivial diagram requires at least $n$ R\IIIa moves. 
\label{thm-infty}
\end{theorem}
\medskip

The rest of the paper is organized as follows. 
In Section~\ref{section-OU-sequence}, we define the non-self OU sequence of a link diagram. 
In Section~\ref{section-core-bridge}, we introduce the OU number and discuss its properties. 
In Section~\ref{section-r3}, we prove Theorems \ref{thm-r3a}--\ref{thm-unordered}. 
In Section~\ref{section-example}, we see examples and prove Theorem \ref{thm-infty}.

\section{The non-self OU sequence of link diagrams}
\label{section-OU-sequence}

In this section, we introduce the definition of the non-self OU sequence for link diagrams. 

\begin{definition}
Let $D=K_1 \cup K_2 \cup \dots \cup K_r$ be a diagram on $S^2$ of an oriented ordered $r$-component link $L$. 
When traversing a link diagram, we encounter crossings either as an over-crossing (denoted by ``$O$'') or an under-crossing (denoted by ``$U$''). 
By focusing on the non-self crossings, we define the non-self OU sequence. 
By traversing $K_i$ with its given orientation and recording the over/under information of {\it only the non-self crossings} that are encountered, we obtain a cyclic sequence, called the {\it non-self OU sequence of $K_i$} and denoted by $f(K_i)$. 
An example is shown in Figure \ref{fig-OU-l}. 
\end{definition}
\begin{figure}[ht]
\centering
\includegraphics[width=2.7cm]{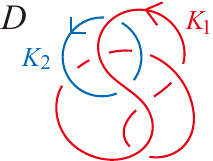}
\caption{A 2-component link diagram $D=K_1 \cup K_2$ with non-self OU sequences $f(K_1)=O^2 U^2$ and $f(K_2)=OUOU$. Note that self crossings are not counted.}
\label{fig-OU-l}
\end{figure}

\medskip 
\begin{proposition}
The length of a non-self OU sequence $f(K_i)$ is an even number for each component $K_i$ of a link diagram $D=K_1 \cup K_2 \cup \dots \cup K_r$. 
\label{prop-card}
\end{proposition}
\medskip 

\begin{proof}
Apply a checkerboard coloring to $K_i$. 
Then the set of all the non-self crossings on $K_i$ is equivalent to the set of all the intersection points between the boundaries of the shaded regions of $K_i$ and the other components. 
Since each shaded region has an even number of intersection points with other components on the boundary, the cardinality of the set is an even number. 
Hence, $K_i$ has an even number of non-self crossings, and therefore the length of $f(K_i)$ is an even number. 
\end{proof}
\medskip

\noindent The following lemma describes a relation between the number of $O$'s and the linking number. 

\medskip 
\begin{lemma}
Let $K_i$ be a component of an oriented link diagram $D=K_1 \cup K_2 \cup \dots \cup K_r$. 
Let $\# O( f(K_i))$ \textup{(}resp. $\# U( f(K_i))$\textup{)} be the number of occurrences of $O$ \textup{(}resp. $U$\textup{)} in $f(K_i)$. 
Then
\begin{itemize}
\item[$\bullet$] $\displaystyle \left\lvert \sum_{j \neq i} lk (K_i, K_j) \right\rvert \leq \min \left\{ \# O( f(K_i)), \ \# U( f(K_i)) \right\}$ and 
\item[$\bullet$] $\displaystyle \sum_{j \neq i} lk (K_i, K_j) \equiv \# O( f(K_i)) \pmod{2}$.
\end{itemize}
\label{lem-lk-mod}
\end{lemma}
\medskip 

\begin{proof}
Let $m$ be the number of non-self crossings on $K_i$. 
Then $m= \# O( f(K_i)) + \# U( f(K_i))$ since $K_i$ has $\# O( f(K_i))$ non-self over-crossings and $\# U( f(K_i))$ non-self under-crossings. 
Transform $D$ into a link diagram $D' =K'_1 \cup K'_2 \cup \dots \cup K'_r$ such that $f(K'_i)=U^m$ (resp.~$O^m$) by $\# O( f(K_i))$ (resp.~$\# U( f(K_i))$) crossing changes at the non-self over-crossings (resp.~under-crossings) of $K_i$. 
In $D'$, the component $K'_i$ is under-most (resp.~over-most) since every non-self crossing on $K'_i$ is an under-crossing (resp.~over-crossing). 
We have $\sum_{j \neq i} lk (K'_i, K'_j)=0$ because the component $K'_i$ is split from the other components. 
Since each crossing change at a non-self crossing shifts the linking number by $+1$ or $-1$, we have the inequality and the congruence. 
Note that $\# U( f(K_i)) = m- \# O( f(K_i)) \equiv \# O( f(K_i)) \pmod{2}$ since $m$ is even by Proposition \ref{prop-card}. 
\end{proof}

\section{The OU number}
\label{section-core-bridge}

In this section, we introduce the OU number.

\subsection{Definition of the OU number for sequences}
\label{subsection-sigma-sequence}

In this subsection, we define the OU number for non-cyclic and cyclic sequences of $O$ and $U$.  

\medskip 
\begin{definition}
Let $S=X_1 X_2 \cdots X_m$ be a non-cyclic sequence, where $X_i=O$ or $U$ ($i=1, 2, \dots , m$). 
We give the sign $\varepsilon$ for each letter $X_i$ in $S$ in the following manner. 
\begin{align*}
\varepsilon (X_i) =  \left\{
\begin{array}{ll}
+1 & \text{ (if } X_i=O \text{ and } i \text{ is even or } X_i=U \text{ and } i \text{ is odd)} \\
-1 & \text{ (if } X_i=O \text{ and } i \text{ is odd or } X_i=U \text{ and } i \text{ is even)}
\end{array}
\right.
\end{align*}
Let $\Phi (S)= \frac{1}{2} \sum_{i=1}^m \varepsilon (X_i)$. 
That is, the value of $\Phi (S)$ is half the number of occurrences of $O$ in even positions and $U$ in odd positions minus the number of occurrences of $O$ in odd positions and $U$ in even positions. 
We call $\Phi (S)$ the {\it OU number of $S$}. 
\end{definition}
\medskip

\begin{example}
For the sequence $S=OU^3OU$, we have the OU number $\Phi (S)=\Phi (OUUUOU)=\frac{1}{2}(1-5)=-2$ since $S$ has one $U$ in an odd position, two $O$'s in odd positions, and three $U$'s in even positions. 
\label{ex-S}
\end{example}
\medskip

\noindent We call the sequence $OUO \cdots U$ an {\it alternating sequence}. 
We have the following formula. 

\medskip 
\begin{example}
Let $S=OUO \cdots U$ be an alternating sequence of an even length $m$. 
Then we have $\Phi (S)= -\frac{m}{2}$. 
\label{ex-sigma-m}
\end{example}
\medskip 

\noindent We have the following proposition regarding the reverse operation. 

\medskip 
\begin{proposition}
Let $S$ be a sequence of $O$ and $U$ with length $m$. 
Let $S'$ be the sequence obtained from $S$ by reversing the order of letters. 
Then $\Phi (S')=-(-1)^m \Phi (S)$. 
\label{prop-sigma-reversed}
\end{proposition}
\medskip 

\begin{proof}
When $m$ is odd, each pair of corresponding letters in $S$ and $S'$ has the same sign, 
since the corresponding positions have the same parity.
When $m$ is even, each pair has opposite signs, since the corresponding positions have opposite parity. 
\end{proof}
\medskip 

\noindent We have the following proposition regarding the rotation. 

\medskip 
\begin{proposition}
Let $S=XY \cdots Z$ be a sequence of $O$ and $U$ of even length $m$. 
Let $S'=Y \cdots ZX$ be the sequence obtained from $S$ by a rotation moving the initial letter $X$ to the terminal position. 
Then $\Phi (S')=- \Phi (S)$. 
\label{prop-rot}
\end{proposition}
\medskip 

\begin{proof}
The signs of corresponding letters in $S$ and $S'$ are opposite because their positions are shifted by either 1 or $m-1$. 
\end{proof}
\medskip 

\noindent Note that Proposition \ref{prop-rot} does not hold when $m$ is odd. 
For example, we have $\Phi (OUOUO)= - \frac{5}{2}$ and $\Phi (UOUOO)= \frac{3}{2}$. 
For cyclic sequences of even length, we define the OU number as follows. 

\medskip 
\begin{definition}
Let $w$ be a cyclic sequence of $O$ and $U$ of even length. 
We define the {\it OU number $\Phi (w)$ of $w$} by $| \Phi (S) |$, where $S$ is a non-cyclic sequence that is obtained from $w$. 
\end{definition}
\medskip 

\noindent Note that $\Phi (w)$ is well-defined by Proposition \ref{prop-rot}. 

\medskip 
\begin{example}
For a cyclic sequence $w=OU^3OU$, we have $\Phi (w)=2$. 
\end{example}
\medskip 

\begin{example}
Let $w=OUO \cdots U$ be an alternating cyclic sequence of an even length $m$. 
By Example \ref{ex-sigma-m}, we have $\Phi (w)= \frac{m}{2}$. 
\end{example}
\medskip 

\noindent By definition, we have the following proposition. 

\medskip 
\begin{proposition}
Let $w$ be a cyclic sequence of $O$ and $U$ with an even length $m$. 
Then $\Phi (w) \leq \frac{m}{2}$. 
The equality holds if and only if $w$ is an alternating sequence. 
\label{prop-OU-length}
\end{proposition}
\medskip 

\noindent The following proposition will be used in Section~\ref{section-r3}. 

\medskip 
\begin{proposition}
Let $w$ be a cyclic sequence of $O$ and $U$ of even length that has at least one pair ``$OU$''. 
Let $w'$ be a sequence that is obtained from $w$ by replacing one pair ``$OU$'' with ``$UO$''. 
Then $|\Phi (w')- \Phi (w) | =0$ or $2$. 
In particular, we have $|\Phi (w')- \Phi (w) | =2$ when $\Phi (w) \neq 1$. 
\label{prop-ou-uo}
\end{proposition}
\medskip 

\begin{proof}
Let $S= X_1 X_2 \cdots X_m$ be a non-cyclic sequence obtained from $w$ that has at least one pair ``$OU$''. 
Suppose that ``$X_i X_{i+1}$'' is ``$OU$''. 
Replace this ``$OU$'' with ``$UO$''. 
Then the pair of signs of $(X_i, X_{i+1})$ is changed from $(+1, +1)$ to $(-1, -1)$ (resp.~from $(-1, -1)$ to $(+1, +1)$) when $i$ is even (resp.~odd). 
Hence, $\Phi (S')= \Phi (S) \pm 2$ for the obtained sequence $S'$. 
Then we have 
$$\Phi (w')- \Phi (w)= | \Phi (S')|- |\Phi (S)|= |\Phi (S) \pm 2 | - | \Phi (S)|.$$
Note that the value of $\Phi (S)$ is an integer since the length of $S$ is even. 
\begin{itemize}
\item[$\bullet$] Suppose that $\Phi (S)= \pm 1$. Then $|\Phi (S) \pm 2 | - | \Phi (S)| = 0$ or $2$. 
\item[$\bullet$] Suppose that $\Phi (S)= 0$. Then $|\Phi (S) \pm 2 | - | \Phi (S)| = 2$. 
\item[$\bullet$] Suppose that $\Phi (S) \leq -2$. Then $|\Phi (S) \pm 2 | - | \Phi (S)| = -(\Phi (S) \pm 2)  - (- \Phi (S)) = \mp 2$. 
\item[$\bullet$] Suppose that $\Phi (S) \geq 2$. Then $|\Phi (S) \pm 2 | - | \Phi (S)| = (\Phi (S) \pm 2)  -  \Phi (S) = \pm 2$. 
\end{itemize}
Hence, we have $|\Phi (w') - \Phi (w)|=2$ when $\Phi (w) \neq 1$ and $|\Phi (w') - \Phi (w)|=0$ or $2$ when $\Phi (w) = 1$. 
\end{proof}
\medskip 

\noindent For the case of $\Phi (w) =1$, the following example shows both cases of $| \Phi (w') -\Phi (w) | =0$ and $2$. 

\medskip 
\begin{example}
For the pair of non-cyclic sequences $S=(OU)OO$ and $S'=(UO)OO$, we have $\Phi (S)=-1$ and $\Phi (S')=1$, and $\Phi (w)= \Phi (w') =1$ for the corresponding cyclic sequences. 
For the pair of non-cyclic sequences $T=(OU)UOUO$ and $T'=(UO)UOUO$, we have $\Phi (T)=1$ and $\Phi (T')=3$, and then $\Phi (v)=1$ and $|\Phi (v') - \Phi (v)|=2$ for the corresponding cyclic sequences. 
\end{example}
\medskip

\subsection{Reduction}
\label{subsection-contraction}

In this subsection, we introduce an operation, the ``reduction'' for sequences of $O$ and $U$ to make the calculation of $\Phi$ easier. 

\medskip 
\begin{definition}
For a (cyclic or non-cyclic) sequence of $O$ and $U$, a {\it reduction} is an operation removing two consecutive letters ``$OO$'' or ``$UU$''. 
\end{definition}
\medskip 

\begin{example}
We obtain $OUOU$ from $OUUUOU$ by a reduction. 
\end{example}
\medskip 

\begin{proposition}
Let $S'$ be a non-cyclic sequence of $O$ and $U$ that is obtained from a non-cyclic sequence $S$ by a reduction. 
Then $\Phi (S')= \Phi (S)$. 
\label{prop-con-sigma}
\end{proposition}
\medskip 

\begin{proof}
Let $S=X_1 X_2 \cdots X_m$. 
Suppose that $S'$ is obtained from $S$ by a reduction removing $X_i X_{i+1}$. 
The signs of $X_i$ and $X_{i+1}$ in $S$ are opposite since they are the same letter $O$ or $U$ and have opposite parities of the positions. 
For other letters, the signs of $X_j$ ($j<i$ or $i+1 <j$) are the same in $S$ and $S'$ because the parity of its position is preserved by the reduction. 
Hence, we obtain $\Phi (S')= \Phi (S)$. 
\end{proof}
\medskip 

\noindent Note that the parity of the length of a sequence is preserved under reduction. 
The following corollary follows from Proposition \ref{prop-con-sigma}. 

\medskip 
\begin{corollary}
Let $w'$ be a cyclic sequence of $O$ and $U$ that is obtained from a cyclic sequence $w$ of even length by a reduction. 
Then $\Phi (w')= \Phi (w)$. 
\label{cor-con-sigma}
\end{corollary}
\medskip 

\noindent Proposition \ref{prop-con-sigma} and Corollary \ref{cor-con-sigma} make the calculation of the OU number easier as shown in the following example. 

\medskip 
\begin{example}
By reductions, we can compute the OU number as $$\Phi (OOUOOUOU)=\Phi (UOOUOU)= \Phi (UUOU)= \Phi (OU)=1.$$
\end{example}
\medskip 

\noindent By applying reductions repeatedly (in any order) to a cyclic sequence $w$ of even length, we obtain an alternating sequence of even length, and the number of the pairs ``$OU$'' coincides with the value of the OU number $\Phi (w)$.

\subsection{The OU number of link diagrams}
\label{subsection-sigma-link}

We define the OU number for link diagrams.

\medskip 
\begin{definition}
Let $D=K_1 \cup K_2 \cup \dots \cup K_r$ be a diagram of an oriented ordered $r$-component link. 
Let $f(K_i)$ be the non-self OU sequence of $K_i$ ($i=1, 2, \dots , r$). 
We define the {\it OU number of $K_i$} as $\Phi (K_i)= \Phi (f(K_i))$. 
\end{definition}
\medskip 

\noindent When a component $K_i$ has no non-self crossings, we define $f(K_i)= \emptyset$, the empty sequence, with OU number zero. 
Note that $\Phi (K_i)$ is well-defined by Propositions \ref{prop-card} and \ref{prop-rot}. 
The value of $\Phi (K_i)$ is a non-negative integer. 

\medskip 
\begin{example}
The diagram $D=K_1 \cup K_2 \cup K_3$ of a trivial 3-component link shown in Figure \ref{fig-3} has $\Phi (K_1)=0$, $\Phi (K_2)=2$ and $\Phi (K_3)=0$. 
\begin{figure}[ht]
\centering
\includegraphics[width=3cm]{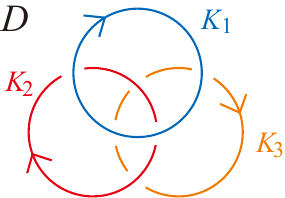}
\caption{A diagram $D=K_1 \cup K_2 \cup K_3$ with $f(K_1)=O^4$, $f(K_2)=OUOU$ and $f(K_3)=U^4$. }
\label{fig-3}
\end{figure}
\end{example}
\medskip 

\begin{example}
For a trivial diagram $D=K_1 \cup K_2 \cup \dots \cup K_r$, we have $\Phi (K_i)=0$ for each component $K_i$ since $f(K_i)$ is the empty sequence $\emptyset$. 
\end{example}
\medskip 

\noindent For orientation, the following proposition holds. 

\medskip 
\begin{proposition}
Let $D=K_1 \cup K_2 \cup \dots \cup K_r$ be a diagram of an oriented link. 
Let $D'=K'_1 \cup K'_2 \cup \dots \cup K'_r$ be the diagram obtained from $D$ by reversing the orientation of a component $K_j$. 
Then $\Phi (K'_i)=\Phi (K_i)$ for any $i$. 
\label{prop-sigma-ori}
\end{proposition}
\medskip 

\begin{proof}
When $i \neq j$, we have $f(K'_i)= f(K_i)$. 
When $i=j$, $f(K'_i)$ is obtained from $f(K_j)$ by reversing the order of letters. 
We have $\Phi (K'_i)= \Phi (K_i)$ by Proposition \ref{prop-sigma-reversed}. 
\end{proof}
\medskip 

\noindent By Proposition \ref{prop-sigma-ori}, we can define the OU number $\Phi (K_i)$ for unoriented link diagrams $D=K_1 \cup K_2 \cup \dots \cup K_r$ as well. 
As a relation to the linking number, we have the following corollary. 

\medskip 
\begin{corollary}
Let $D=K_1 \cup K_2 \cup \dots \cup K_r$ be an oriented link diagram. 
We have $\Phi (K_i) \equiv \sum_{j \neq i} lk (K_i, K_j) \pmod{2}.$ 
That is, the parity of $\Phi (K_i)$ coincides with the parity of the total linking number of $K_i$ with other components. 
\label{cor-lk}
\end{corollary}
\medskip 

\begin{proof}
Let $w'$ be an alternating sequence that is obtained from $f(K_i)$ by reductions. 
Since the number of $O$'s in $w'$, which equals $\Phi (K_i)$, is congruent to that in $f(K_i)$ modulo 2, we obtain the congruence by Lemma \ref{lem-lk-mod}. 
\end{proof}

\section{The OU number and Reidemeister moves}
\label{section-r3}

In this section, we discuss the behavior of the OU number under Reidemeister moves. 
The OU number is unchanged by an RI move since the non-self OU sequence is not affected by an RI move, which involves only a self-crossing. 
For R\II moves, the following proposition holds. 

\medskip 
\begin{proposition}
For a link diagram $D=K_1 \cup K_2 \cup \dots \cup K_r$, 
the OU number $\Phi (K_i)$ of each component $K_i$ is invariant under R\II moves. 
\label{prop-r2}
\end{proposition}
\medskip 

\begin{proof}
For an R\II move in the same component, the non-self OU sequences are unchanged, and therefore $\Phi$ is also unchanged for each $K_i$. 
Let $D'=K'_1 \cup K'_2 \cup \dots \cup K'_r$ be a diagram that is obtained from $D=K_1 \cup K_2 \cup \dots \cup K_r$ by an R\II move between $K_i$ and $K_j$ ($i \neq j$) reducing two crossings. 
Then $f(K'_k)$ is obtained from $f(K_k)$ by a reduction ($k=i, j$). 
There are no changes of non-self OU sequences for other components. 
Hence, the OU numbers are unchanged for all components. 
\end{proof}
\medskip 

\noindent Now we discuss the behavior of the OU number $\Phi$ under R\III moves. 
For the three segments where an R\III move is applied, we call them the {\it upper}, {\it middle}, and {\it lower segments} according to the layeredness as shown in Figure~\ref{fig-R3}. 
\begin{figure}[ht]
\centering
\includegraphics[width=9cm]{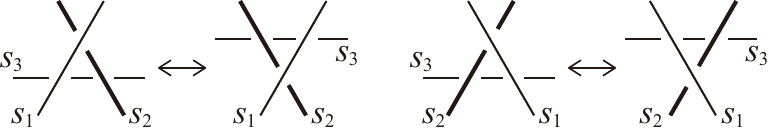}
\caption{The upper, middle, and lower segments $s_1, s_2$, and $s_3$. Middle segments are thickened. }
\label{fig-R3}
\end{figure}

\medskip 
\begin{definition}
For an R\III move, when the upper and middle segments belong to different components and the middle and lower segments also belong to different components, we call the R\III move the {\it Reidemeister move of type \IIIa}, or simply \emph{R\IIIa}. 
(The upper and lower segments may belong to the same or different components.) 
\label{def-r3a}
\end{definition}
\medskip

\noindent 
In other words, an R\III move is an R\IIIa move when both crossings on the middle segment are non-self crossings. 
See Figure \ref{fig-r3-ex} for an example.
The following proposition holds. 

\medskip 
\begin{proposition}
For a link diagram $D=K_1 \cup K_2 \cup \dots \cup K_r$, 
the OU number $\Phi (K_i)$ of each component $K_i$ is invariant under R\III moves except for R\IIIa. 
\label{prop-r3}
\end{proposition}
\medskip 

\begin{proof}
When an R\III move is applied, there are no changes of the non-self OU sequences at the upper and lower segments. 
For the middle segment, there is a change of ``$OU$'' and ``$UO$'' only when both crossings on the middle segment are non-self crossings, namely, when the R\III move is an R\IIIa move. 
\end{proof}

We are now ready to prove Theorem \ref{thm-r3a}.

\begin{proof}[Proof of Theorem~\ref{thm-r3a}]
It immediately follows from Propositions \ref{prop-r2} and \ref{prop-r3}. 
\end{proof}

For unordered links, we have the following corollary. 

\begin{corollary}
Let $D=K_1 \cup K_2 \cup \dots \cup K_r$ and $D'=K'_1 \cup K'_2 \cup \dots \cup K'_r$ be diagrams representing the same unoriented unordered link. 
If 
\begin{align*}
\{ \Phi (K_1), \Phi (K_2), \dots , \Phi (K_r) \} \neq \{ \Phi (K'_1), \Phi (K'_2), \dots , \Phi (K'_r) \}
\end{align*}
as multisets, then any sequence of Reidemeister moves between $D$ and $D'$ must contain an R\IIIa move. 
\end{corollary}
\medskip

\begin{example}
To transform the diagram illustrated in Figure \ref{fig-3} into a trivial diagram, at least one R\IIIa move is required since $\{ \Phi (K_1), \Phi (K_2), \Phi (K_3) \} = \{ 0, 2, 0 \} \neq \{ 0, 0, 0 \}$.\footnote{In \cite{M}, the necessity of an R\III move for this diagram is explained using the checkerboard coloring to a knot component. } 
Indeed, the diagram is transformed into a trivial diagram by exactly one R\IIIa move and R\hspace{-0.5pt}I\hspace{-0.5pt}I moves. 
\end{example}
\medskip 

\begin{proposition}
Let $D'=K'_1 \cup K'_2 \cup  \dots \cup K'_r$ be an unoriented ordered link diagram obtained from $D=K_1 \cup K_2 \cup \dots \cup K_r$ by an R\IIIa move where the middle segment belongs to $K_j$ and $K'_j$. 
Then $| \Phi (K'_j) - \Phi (K_j) | =0$ or $2$ and $\Phi (K'_i) = \Phi (K_i)$ for $i \neq j$. 
In particular, as long as $\Phi (K_j) \neq 1$, we have $| \Phi (K'_j) - \Phi (K_j) | =2$. 
\label{prop-two}
\end{proposition}
\medskip 

\begin{proof}
It follows from the proof of Proposition \ref{prop-r3} and  Proposition \ref{prop-ou-uo}. 
\end{proof}
\medskip 

\noindent The following corollary follows from Proposition~\ref{prop-two}. 

\medskip
\begin{corollary}
Let $D=K_1 \cup K_2 \cup \dots \cup K_r$, $D'=K'_1 \cup K'_2 \cup \dots \cup K'_r$ be diagrams representing the same unoriented ordered link. 
If $| \Phi (K'_j) - \Phi (K_j)|=N$, then any sequence of Reidemeister moves between $D$ and $D'$ must contain at least $\frac{N}{2}$ R\IIIa moves whose middle segment belongs to the $j$th component. 
\label{cor-r3-ki}
\end{corollary}

\begin{proof}
Since such an R\IIIa move changes the OU number by at most 2 for the $j$th component, at least $\frac{N}{2}$ moves are needed. 
\end{proof}
\medskip 

We prove Theorems \ref{thm-ordered} and \ref{thm-unordered}.

\begin{proof}[Proof of Theorems~\ref{thm-ordered} and~\ref{thm-unordered}]
The inequalities follow from Corollary \ref{cor-r3-ki}. 
\end{proof}

\section{Examples}
\label{section-example}

In this section, we observe examples and prove Theorem \ref{thm-infty}.

\subsection{Jablonowski's diagram}

\noindent Let $D_{n,k}$ be the diagram of a trivial 2-component link shown in Figure \ref{fig-d-nk}. 

\begin{figure}[ht]
\centering
\includegraphics[width=10cm]{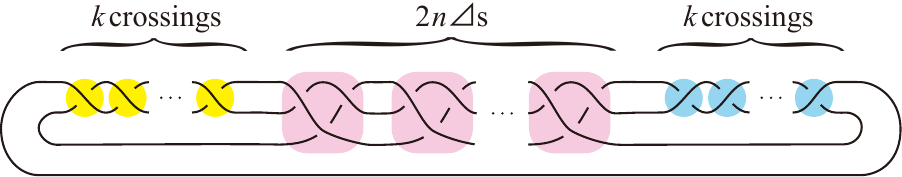}
\caption{Jablonowski's diagram $D_{n,k}$. }
\label{fig-d-nk}
\end{figure}

\noindent 
Jablonowski \cite{J} proved that transforming $D_{n,k}$ into a trivial diagram requires at least four R\III moves for any $n \geq 2$ and odd $k \geq 3$. 
We have the following.

\medskip 
\begin{proposition}
For any $n \geq 2$ and odd $k \geq 3$, transforming the diagram $D_{n,k}$ into a trivial diagram requires at least $n$ R\IIIa moves.\footnote{At the moment, any sequence of Reidemeister moves from $D_{n,k}$ to the trivial diagram which has exactly $n$ R\IIIa moves has not been found. }
\label{prop-dnk}
\end{proposition}
\medskip 

\begin{proof}
Give an orientation and base points to $D_{n,k}$ as shown in Figure \ref{fig-d-pf}. 
By traversing $K_1$ and $K_2$ from the base points, we obtain their non-self OU sequences with the following expressions. 
\begin{align*}
f(K_1)= & \overbrace{OUO \cdots O}^{k} \ \overbrace{UOUO\cdots UO}^{4n} \ \overbrace{OUO \cdots O}^{k}, \\
f(K_2)= & \overbrace{OU \cdots OU}^{2n} \ \overbrace{UOU\cdots U}^{k} \ \overbrace{OU \cdots OU}^{2n} \ \overbrace{UOU \cdots U}^{k}.
\end{align*}
When $4n>k$, the non-self OU sequence $f(K_1)$ can be reduced by removing the $2k$ letters from the right-hand side in the above expression, and an alternating sequence of length $4n$ is obtained. 
Hence, we have $\Phi (K_1)=2n$. 
When $4n \leq k$, $f(K_1)$ can be reduced by $8n$ letters and then $2k-8n$ letters, and an alternating sequence of length $4n$ is obtained. 
We have $\Phi (K_1) = 2n$. 
The sequence $f(K_2)$ can be reduced into $\emptyset$ and we obtain $\Phi (K_2)=0$. 
Therefore, at least $\frac{1}{2} | (2n+0)-(0+0) | = n$ R\IIIa moves are needed to transform $D_{n, k}$ into the trivial diagram by Theorem~\ref{thm-unordered}. 
\end{proof}

\begin{figure}[ht]
\centering
\includegraphics[width=10cm]{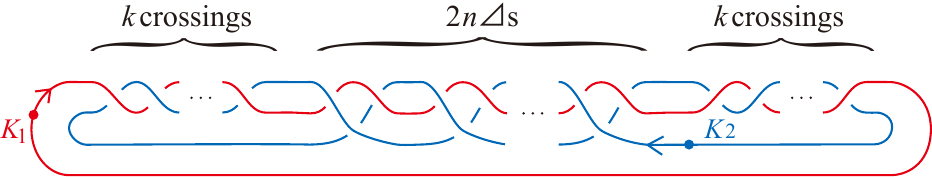}
\caption{Jablonowski's diagram $D_{n,k}$ with base points. }
\label{fig-d-pf}
\end{figure}

We prove Theorem \ref{thm-infty}.

\begin{proof}[Proof of Theorem~\ref{thm-infty}]
For $n \geq 2$, it follows from Proposition \ref{prop-dnk}. 
For $n=1$, it follows by taking $D_{2,k}$. 
\end{proof}
\medskip 

\noindent We have the following corollary. 

\medskip 
\begin{corollary}
For any link $L$ with at least two components and any natural number $n$, there is a pair of diagrams $D^1$ and $D^2$ of $L$ such that any sequence of Reidemeister moves transforming $D^1$ into $D^2$ must contain at least $n$ R\IIIa moves. 
\label{cor-r3}
\end{corollary}

\begin{proof}
Let $D^1=K^1_1 \cup K^1_2 \cup \dots \cup K^1_r$ be a diagram of an oriented $r$-component link $L$, in which two components $K^1_i$ and $K^1_j$ have at least one non-self crossing. 
Suppose that $K^1_i$ has $2m$ non-self crossings in $D^1$. 
Then the length of $f(K^1_i)$ is $2m$ and $\Phi (K^1_i) \leq m$ by Proposition \ref{prop-OU-length}. 
Take a region $R$ of $D^1$ such that $R$ has a crossing between $K^1_i$ and $K^1_j$ on the boundary. 
Let $D=D_{m+n, 3}$ be the Jablonowski's diagram. 
Put $D$ inside the region $R$. 
Take small arcs of $K^1_i$ and $K^1_j$ on $\partial R$ and connect $D$ to $D^1$ as shown in Figure \ref{fig-d1d2}. 
\begin{figure}[ht]
\centering
\includegraphics[width=10cm]{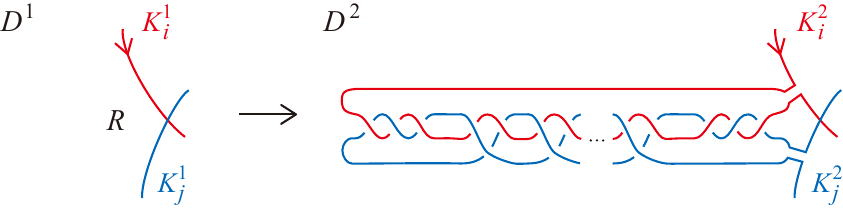}
\caption{Diagrams $D^1$ and $D^2$.} 
\label{fig-d1d2}
\end{figure}
Let $D^2=K^2_1 \cup K^2_2 \cup \dots \cup K^2_r$ be the obtained diagram. 
By the proof of Proposition~\ref{prop-dnk}, the non-self OU sequence $f(K^2_i)$ can be reduced to a sequence $f(K^1_i)w_a$, where $w_a$ is an alternating sequence of length $4m+4n$. 
By further reductions, we obtain an alternating sequence of length $2m+4n$ or more since the length of $f(K^1_i)$ is $2m$. 
Hence, we have $\Phi (K^2_i) \geq m+2n$ by Proposition \ref{prop-OU-length}. 
For other components including $K^1_j$, the OU numbers are unchanged, namely, $\Phi (K^2_l)= \Phi (K^1_l)$ for any $l \neq i$. 
Hence, we obtain a lower bound for the number of necessary R\IIIa moves that transform $D^2$ into $D^1$ as $\frac{1}{2} | \Phi (K^2_i)- \Phi (K^1_i) | \geq \frac{1}{2} | (m+2n)-m | =n$ by Theorems \ref{thm-ordered}, \ref{thm-unordered} for both ordered, unordered cases. 
\end{proof}

\subsection{Hayashi, Hayashi, and Nowik's diagram}

Let $D_n$ be a diagram of a trivial 2-component link that is the closure of an $(n+1)$-braid diagram 
$$\sigma_1 (\sigma_2 \sigma_1)(\sigma_3 \sigma_2) \dots (\sigma_n \sigma_{n-1}) \sigma_n^{-n}.$$ 
C. Hayashi, M. Hayashi, and T. Nowik \cite{HH} showed that for any natural number $n$, the diagram $D_n$ can be deformed into a trivial diagram by a sequence of $\frac{n^2+3n-2}{2}$ Reidemeister moves that consists of $(n-1)$ RI moves reducing a crossing, $n$ R\II moves reducing an incoherent bigon, and $\frac{(n-1)n}{2}$ R\III moves. 
They also gave a lower bound for the number of necessary Reidemeister moves deforming $D_n$ into a trivial diagram. 

\begin{figure}[ht]
\centering
\includegraphics[width=10cm]{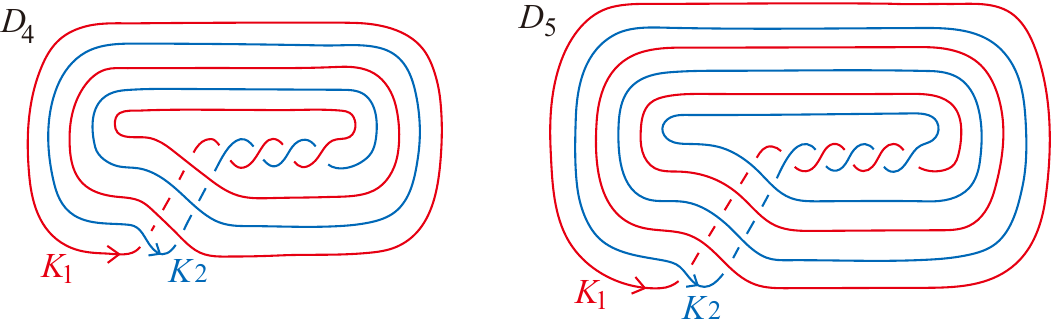}
\caption{Diagrams $D_4$ and $D_5$. }
\label{fig-dn}
\end{figure}

We have the following regarding R\IIIa moves. 

\medskip 
\begin{proposition} Let $n \geq 4$. 
Transforming $D_n$ into a trivial diagram requires at least 
\begin{itemize}
\item[$\bullet$] $\frac{n}{2}$ R\IIIa moves when $n \equiv 0, \ 2 \pmod{4}$, 
\item[$\bullet$] $\frac{n-1}{2}$ R\IIIa moves when $n \equiv 1 \pmod{4}$, and 
\item[$\bullet$] $\frac{n+1}{2}$ R\IIIa moves when $n \equiv 3 \pmod{4}$. 
\end{itemize}
\end{proposition}
\medskip 

\begin{proof}
Give the orientation and the order of components as shown in Figure~\ref{fig-dn}. 
\begin{itemize}
\item[(1)] When $n$ is even, the OU sequences are as follows. 
\begin{align*}
f(K_1)= & \overbrace{UU \cdots U}^{\frac{n}{2}} \ \overbrace{UOU\cdots UO}^{n} \ \overbrace{OO \cdots O}^{\frac{n}{2}}, \\
f(K_2)= & \overbrace{UU \cdots U}^{\frac{n}{2}} \ \overbrace{OUO \cdots OU}^{n} \ \overbrace{OO \cdots O}^{\frac{n}{2}}.
\end{align*}
\begin{itemize}
\item[$\bullet$] Suppose that $n \equiv 0 \pmod{4}$. 
For $f(K_1)$, the $\frac{n}{2}$ letters of $U$ on the left and the $\frac{n}{2}$ letters of $O$ on the right can be removed by reductions and an alternating sequence of length $n$ is obtained. 
Hence $\Phi (K_1)= \frac{n}{2}$. 
We obtain $\Phi (K_2)= \frac{n}{2}$ in the same way. 
Hence, $\frac{1}{2} ( \Phi (K_1) + \Phi (K_2) -0)= \frac{n}{2}$. 
\item[$\bullet$] Suppose that $n \equiv 2 \pmod{4}$. 
For $f(K_1)$, the $(\frac{n}{2}+1)$ letters of $U$ on the left and the $(\frac{n}{2}+1)$ letters of $O$ on the right can be removed by reductions and an alternating sequence of length $(n-2)$ is obtained. 
Hence $\Phi (K_1)= \frac{n}{2}-1$. 
For $f(K_2)$, the $(\frac{n}{2}-1)$ letters of $U$ on the left and the $(\frac{n}{2}-1)$ letters of $O$ on the right can be removed by reductions and an alternating sequence of length $(n+2)$ is obtained. 
Hence $\Phi (K_2)= \frac{n}{2}+1$. 
Therefore, we obtain $\frac{1}{2} ( \Phi (K_1) + \Phi (K_2))= \frac{n}{2}$. 
\end{itemize}
\item[(2)] When $n$ is odd, the OU sequences are as follows. 
\begin{align*}
f(K_1)= & \overbrace{UU \cdots U}^{\frac{n+1}{2} } \ \overbrace{UOU\cdots OU}^{n} \ \overbrace{OO \cdots O}^{\frac{n-1}{2}}, \\
f(K_2)= & \overbrace{UU \cdots U}^{ \frac{n-1}{2}} \ \overbrace{OUO \cdots UO}^{n} \ \overbrace{OO \cdots O}^{\frac{n+1}{2}}.
\end{align*}
\begin{itemize}
\item[$\bullet$] Suppose that $n \equiv 1 \pmod{4}$. 
For $f(K_1)$, the $(\frac{n+1}{2}+1)$ letters of $U$ on the left and the $\frac{n-1}{2}$ letters of $O$ on the right can be removed by reductions and an alternating sequence of length $(n-1)$ is obtained. 
Hence $\Phi (K_1)= \frac{n-1}{2}$. 
Similarly, we obtain $\Phi (K_2)= \frac{n-1}{2}$. 
Hence, $\frac{1}{2} ( \Phi (K_1) + \Phi (K_2))= \frac{n-1}{2}$. 
\item[$\bullet$] Suppose that $n \equiv 3 \pmod{4}$. 
For $f(K_1)$, the $\frac{n+1}{2}$ letters of $U$ on the left and the $(\frac{n-1}{2}-1)$ letters of $O$ on the right can be removed by reductions and an alternating sequence of length $(n+1)$ is obtained. 
Hence $\Phi (K_1)= \frac{n+1}{2}$. 
Similarly, we obtain $\Phi (K_2)= \frac{n+1}{2}$. 
Hence, $\frac{1}{2} ( \Phi (K_1) + \Phi (K_2))= \frac{n+1}{2}$. 
\end{itemize}
\end{itemize}
\end{proof}
\medskip

\subsection{Reidemeister moves in \texorpdfstring{$\mathbb{R}^2$}{R\textasciicircum 2}}

Let $D$, $D'$ be the diagrams shown in Figure \ref{fig-r2} representing the same 2-component link. 
In \cite{CESS}, Carter, Elhamdadi, Saito and Satoh proved that any sequence of Reidemeister moves transforming $D$ into $D'$ in $\mathbb{R}^2$ requires three R\III moves. 
\begin{figure}[ht]
\centering
\includegraphics[width=10cm]{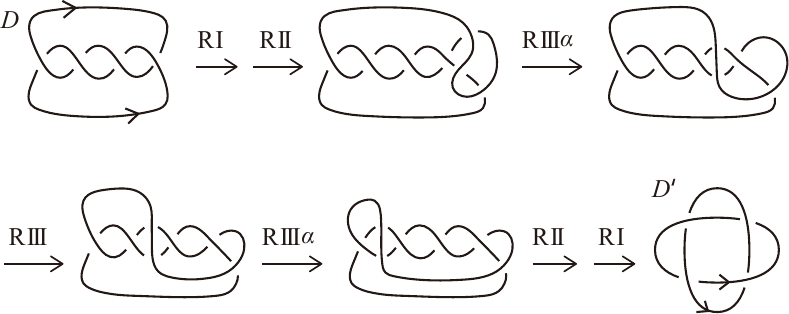}
\caption{A sequence of Reidemeister moves in $\mathbb{R}^2$ given in \cite{CESS}.}
\label{fig-r2}
\end{figure}
Note that $D$ and $D'$ are the same diagrams in $S^2$. 
By our method of Theorems \ref{thm-r3a}, \ref{thm-ordered} or \ref{thm-unordered}, we cannot distinguish the necessity of R\III moves on $\mathbb{R}^2$ for such $D$ and $D'$ since they have the same non-self OU sequence. 
Regarding the number of R\IIIa moves in $\mathbb{R}^2$, we have the following proposition. 

\medskip 
\begin{proposition}
Let $L=k_1 \cup k_2 \cup \dots \cup k_r$ be an unoriented ordered $r$-component link such that $\sum_{j \neq i} lk(k_i , k_j) \equiv 0 \pmod{2}$ with an orientation for each knot component $k_i$. 
Let $D=K_1 \cup K_2 \cup \dots \cup K_r$, $D'= K'_1 \cup K'_2 \cup \dots \cup K'_r$ be diagrams of $L$ on $S^2$ or $\mathbb{R}^2$ such that $\Phi (K_i)= \Phi (K'_i)$ for each $i$. 
Any sequence of Reidemeister moves deforming $D$ into $D'$ in $S^2$ or $\mathbb{R}^2$ contains a non-negative even number of R\IIIa moves. 
\end{proposition}
\medskip 

\begin{proof}
As we have seen, the value of the OU number is preserved by Reidemeister moves except for R\IIIa moves by Theorem \ref{thm-r3a}. 
By an R\IIIa move, the OU number definitely varies by $\pm 2$ at the component of the middle segment as long as the original OU number is not 1 by Proposition \ref{prop-two}. 
We have $\Phi (K_i) \equiv 0 \pmod{2}$ for each component $K_i$ of $D$ since $\Phi (K_i) \equiv \sum_{j \neq i} lk(K_i, K_j) \pmod{2}$ by Corollary \ref{cor-lk}. 
Hence, each R\IIIa move changes the OU number of one component in this case. 
Then the total number of R\IIIa moves must be an even number. 
\end{proof}
\medskip 

\noindent For example, any sequence of Reidemeister moves between $D$ and $D'$ in Figure~\ref{fig-r2} has a non-negative even number of R\IIIa moves since the linking numbers are 2 and all components of $D$ and $D'$ have the same OU number 2.

\section*{Acknowledgments}
This work was partially supported by the JSPS KAKENHI Grant Number JP21K03263 and JST CREST, Japan, Grant Number JPMJCR25Q3.


\begin{thebibliography}{99}

\bibitem{CESS} S. Carter, M. Elhamdadi, M. Saito and S. Satoh, A lower bound for the number of Reidemeister moves of type III, Topology Appl. {\bf 153} (2006), 2788--2794. 

\bibitem{FNS} Y. Funahashi, Y. Nakanishi and S. Satoh, A note on the OU sequences of a 2-bridge knot, J. Knot Theory Ramifications {\bf 25} (13) (2016), 1671001.

\bibitem{HH} C. Hayashi, M. Hayashi and T. Nowik, Unknotting number and number of Reidemeister moves needed for unlinking, Topology Appl. {\bf 159} (2012), 1467--1474. 

\bibitem{HNSY} R. Higa, Y. Nakanishi, S. Satoh and T. Yamamoto, Crossing information and warping polynomials about the trefoil knot, J. Knot Theory Ramifications {\bf 21} (12) (2012), 1250117.

\bibitem{J} M. Jablonowski, On a surface singular braid monoid, Topology Appl. {\bf 160} (2013), 1773--1780. 

\bibitem{M} V. O. Manturov, Knot Theory, CRC Press (2004). 

\bibitem{MT} D. Michieletto and M. S. Turner, A topologically driven glass in ring polymers. PNAS, {\bf 113} (19) (2016), 5195--5200. 

\bibitem{O} O. P. \"{O}stlund, Invariants of knot diagrams and relations among Reidemeister moves, J. Knot Theory Ramifications {\bf 10} (2001), 1215--1227. 

\bibitem{S} A. Shimizu, The warping degree of a link diagram, Osaka J. Math. {\bf 48} (1) (2011), 209--231. 

\bibitem{SG} J. Smrek and A. Y. Grosberg, Minimal surfaces on unconcatenated polymer rings in melt, ACS Macro Lett. {\bf 5} (6) (2016), 750--754. 


\end{thebibliography}
\end{document}